\documentclass{amsart}
\usepackage{amssymb, amsmath, verbatim}

\theoremstyle{plain}
\newtheorem{theorem}{Theorem}[section]
\newtheorem{lemma}{Lemma}[section]
\newtheorem{proposition}{Proposition}[section]
\newtheorem{corollary}{Corollary}[section]

\theoremstyle{definition}
\newtheorem{definition}{Definition}[section]

\theoremstyle{remark}

\newtheorem{question}{Question}[section]
\newtheorem{claim}{Claim}[section]

\newcommand{\itz}{\begin{itemize}}
\newcommand{\zti}{\end{itemize}}
\newcommand{\enu}{\begin{enumerate}}
\newcommand{\une}{\end{enumerate}}
\newcommand{\ie}{{\it i.e.},}

\newcommand{\card}[1]{\lvert #1\rvert}
\newcommand{\al}{\aleph}
\newcommand{\af}{\alpha}
\newcommand{\om}{\omega}
\newcommand{\ka}{\kappa}
\newcommand{\lm}{\lambda}
\newcommand{\oml}{\om_1}
\newcommand{\alo}{\aleph_0}
\newcommand{\eps}{\varepsilon}
\newcommand{\lapps}[1]{long $#1$-approximation sequence}
\newcommand{\cardfloor}[1]{\left\lfloor #1 \right\rfloor}
\newcommand{\cardhead}[2]{\cardfloor{#2}_{#1}}
\newcommand{\cardterm}[2]{\partial_{#1} #2}
\newcommand{\bighcard}{\mu}
\newcommand{\bigh}{H(\bighcard)}
\newcommand{\elemsub}{\prec}
\newcommand{\integers}{\mathbb{Z}}
\newcommand{\restrict}{\upharpoonright}

\DeclareMathOperator{\Fr}{Fr}
\DeclareMathOperator{\fr}{fr}
\DeclareMathOperator{\ran}{ran}
\newcommand{\subalg}{\leq}
\newcommand{\subrc}{\leq_{\mathrm{rc}}}

\newcommand{\subfree}{\leq_{\mathrm{free}}}

\DeclareMathOperator{\dom}{dom}
\DeclareMathOperator{\id}{id}
\newcommand{\vn}{\varnothing}
\newcommand{\mc}[1]{\mathcal{#1}}
\newcommand{\pushout}{\boxplus}
\newcommand{\istst}{it suffices to show that}
\newcommand{\Istst}{It suffices to show that}
\newcommand{\shiftr}{\triangleright}
\newcommand{\shiftl}{\triangleleft}

\newcommand{\commute}[1][]{
  \mathrel{
    \mathop{
      \vcenter{
        \hbox{\oalign{\noalign{\kern-.3ex}\hfil$\vert$\hfil\cr
              \noalign{\kern-.7ex}
              $\smile$\cr\noalign{\kern-.3ex}}}
      }
    }\displaylimits_{#1}
  }
}

\begin{document}
\title{On the strong Freese-Nation property}
\author{David Milovich}
\email{david.milovich@tamiu.edu}
\urladdr{http://www.tamiu.edu/~dmilovich}
\address{
Dept. of Mathematics and Physics, 
Texas A{\&}M International University, 
5201 University Blvd., Laredo, TX 78041, USA
}
\date{Dec. 23, 2014; revised July 7, 2016}
\begin{abstract}
We show that there is a boolean algebra 
that has the Freese-Nation property (FN) but not the strong
Freese-Nation property (SFN), thus answering a question of
Heindorf and Shapiro. Along the way, we produce some
new characterizations of the FN and SFN in
terms of sequences of elementary submodels.
\end{abstract}

\maketitle

\section{The Freese-Nation property and friends}

\begin{definition}\
\itz
\item 
Given a poset $P$ and a map $f$ from $P$
to the power set of $P$, we say that $f$ is
\emph{interpolating} if, for all pairs $x\leq_P y$,
there exists $z\in[x,y]\cap f(x)\cap f(y)$. 
\item
We say that  poset $P$ has the \emph{Freese-Nation property (FN)} if there is
an interpolating map from  $P$ to 
$[P]^{<\alo}$, the set of finite subsets of $P$.
Such a map is called an \emph{FN map}.
\zti
\end{definition}

When a $P$ is also a boolean algebra, the FN can be
understood as an abstraction of the Interpolation
Theorem of propositional logic, which states that
if the implication $\varphi\rightarrow\psi$
is tautological for two propositions $\varphi$
and $\psi$, then there is a proposition $\chi$
such that $\varphi\rightarrow\chi$ and $\chi\rightarrow\psi$
are tautological and the propositional variables of $\chi$
are common to $\varphi$ and $\psi$. An easy 
consequence of the Interpolation Theorem is that
free boolean algebras have the FN.

The FN is named after
Freese and Nation~\cite{freesenation}, 
who introduced it in 1978 as part of 
a characterization of projective lattices.
In particular, every projective lattice has the FN
(but the converse was already known to be false, 
even for finite lattices).
Since the morphisms in the category
of lattices and lattice homomorphisms
are epimorphisms if and only if they are surjective,
a lattice is projective if and only if it is a retract
of a free lattice. Likewise, a boolean algebra is
projective if and only if it is a retract of a free
boolean algebra. The Stone duals of the projective
boolean algebras are exactly the Dugundji spaces,
\ie\ the continuous retracts of powers of 2.

The Stone duals of the boolean algebras with the
FN were elegantly characterized in two ways 
by \v S\v cepin~\cite{sckmet,scfunct},
as the existence of a distance function between 
points and regular closed sets and as the existence
of a rich family of open quotient maps.
Succinctly, a compact Hausdorff space is ``$k$-metrizable''
if and only if it is ``openly generated;''
a boolean space is openly generated if and only if 
its clopen algebra has the FN.
\v S\v cepin also proved that every Dugundji space
is openly generated, that the Vietoris hyperspace operation
preserves open generation, and that every openly
generated boolean space of weight at most $\al_1$
is Dugundji. However, Shapiro~\cite{shapiro} proved that the 
Vietoris hyperspace of $2^\ka$ is not a continuous
image of a power of 2 if $\ka\geq\al_2$.
Thus, for boolean algebras up to size $\al_1$, 
the FN is equivalent to projectivity, while
for boolean algebras in general, projectively
strictly implies the FN.

Fuchino translated \v S\v cepin's notion
of openly generated into the language of elementary
substructures in an appendix to \cite{hs}.
Before we can state this characterization,
we need a few definitions.

\begin{definition}\
\itz
\item If $P$ is a poset, $S$ is a set, and $p\in P$,
then, when they exist, let
 \itz
 \item $\pi_+^S(p)=\min\{q\in P\cap S:q\geq p\}$ and
 \item $\pi_-^S(p)=\min\{q\in P\cap S:q\leq p\}$.
 \zti
\item
Given a poset $P$ and $Q\subseteq P$,
we say that $Q$ is a \emph{relatively complete} suborder of $P$
if, for all $p\in P$, $\pi_+^Q(p),\pi_-^Q(p)$ exist.
\item 
Given boolean algebras $A$ and $B$, we write $A\subalg B$
to indicate that $A$ is a subalgebra of $B$.
\item
If $B$ is a boolean algebra, $A\subalg B$,
and $A$ is a relatively complete suborder of $B$, 
then we write $A\subrc B$.
\item 
If $\psi\colon A\rightarrow B$ is a boolean homomorphism,
we say that $\psi$ is relatively complete if $\psi[A]\subrc B$.
\zti
\end{definition}

Note that the topological dual of a relatively complete boolean
homomorphism is an open map between two boolean spaces.

\begin{definition}\
\itz
\item
Given two sets $P$ and $Q$, we write $P\elemsub Q$ if
$(P,\in)$ is an elementary substructure of $(Q,\in)$.
\item 
Given a cardinal $\mu$, let $H(\mu)$ denote the set 
of all sets with transitive closure smaller than $\mu$.
\zti
\end{definition}

Given a boolean algebra 
$\mc{A}=(A,0,1,\wedge,\vee,{-})$, 
we will abuse notation by using $A$ to denote both $A$ 
and $\mc{A}$. In particular, when we write $A\in M$ 
for some set $M$, we mean $\mc{A}\in M$.

\begin{theorem}[Fuchino]\label{fuchinochar}
Let $A$ be a boolean algebra and let $\bighcard$
be a regular uncountable cardinal such that $A\in\bigh$.
The following are then equivalent.
\enu
\item
$A$ has the FN.
\item 
$A\cap M\subrc A$ for all countable $M$ satisfying $A\in M\elemsub\bigh$.
\item 
$A\cap M\subrc A$ for all $M$ satisfying $A\in M\elemsub\bigh$.
\une
\end{theorem}

If we weaken the definition of FN map to allow as outputs
countable sets instead of merely finite sets, then we obtain
the weak Freese-Nation property (WFN), which was
initially investigated in topological terms by
\v S\v cepin~\cite{scfunct} and later systematically
studied in Heindorf and Shapiro's 1994 book {\it Nearly Projective
Boolean Algebras}~\cite{hs}. For our purposes, their
most interesting result about the WFN is a characterization
of it as the existence of a rich family of commuting subalgebras.
Elementary substructure characterizations 
analogous to (2) and (3) from the previous theorem were 
proved by Fuchino, Koppelberg, and Shelah in~\cite{fks} and
by the author in~\cite{mlkfn}, respectively.

\begin{definition}\
\itz
\item
Given a poset $P$ and $A,B\subseteq P$, we say
that $A$ and $B$ \emph{commute}, writing $A \commute B$,
if, for all pairs $(x,y)\in A\times B$, if
$x\leq y$, then $[x,y]\cap A\cap B$ is nonempty, and
if $y\leq x$, then $[y,x]\cap A\cap B$ is nonempty.
\item Given a poset $P$ and $Q\subseteq P$, we say
that $Q\subseteq_\sigma P$ if, for all $p\in P$,
there exist countable sets $L(p),U(p)\subseteq Q$ such that 
 \itz
 \item
 $\{q\in Q:q\leq p\}=\bigcup_{r\in L(p)}\{q\in Q:q\leq r\}$ and
 \item $\{q\in Q:q\geq p\}=\bigcup_{r\in U(p)}\{q\in Q:q\geq r\}$.
 \zti
\zti
\end{definition}
Note that if $A$ and $B$ are subalgebras of a boolean algebra $C$,
then $A\commute B$ if and only if, for all ultrafilters $U$ of $A$
and $V$ of $B$, if $U\cap B=V\cap A$, then $U\cup V$ extends to an
ultrafilter of $C$.

\begin{theorem}
Let $A$ be a boolean algebra and let $\bighcard$
be a regular uncountable cardinal such that $A\in\bigh$.
The following are then equivalent.
\enu
\item $A$ has the WFN, \ie\ there is an interpolating map
from $A$ to $[A]^{<\al_1}$.
\item (Fuchino, Koppelberg, Shelah) $A\cap M\subseteq_\sigma A$
for all $M$ satisfying $A\in M\elemsub\bigh$ and 
$\card{M}=\oml\subseteq M$.
\item $A\cap M\subseteq_\sigma A$
for all $M$ satisfying $A\in M\elemsub\bigh$ and 
$\oml\subseteq M$.
\item (\v S\v cepin, Heindorf, Shapiro) There is a cofinal
family $\mc{C}$ of countable subalgebras of $A$ such that
$F\commute G$ for all $F,G\in\mc{C}$.
\une
\end{theorem}

In~\cite{hs}, Heindorf and Shapiro defined the 
natural analog of (4) for the FN to be the strong Freese-Nation 
property (SFN).

\begin{definition}
A boolean algebra has the \emph{SFN} if and only if it
has a pairwise commuting cofinal family of finite subalgebras.
\end{definition}

Also in~\cite{hs}, Heindorf and Shapiro showed 
that projectivity implies the SFN implies the FN. 
Hence, the three properties are equivalent for 
boolean algebras of size at most $\al_1$.
They further showed that the symmetric square 
and exponential operations preserve the SFN;
\v S\v cepin had already shown the same for
the FN~\cite{sckmet}. 
On the other hand, if $\ka\geq\al_2$, then
the exponential~\cite{shapiro}\
and symmetric square~\cite{scfunct}\ of a
free boolean algebra of size $\ka$
are not projective. Thus, the two most natural 
examples of non-projective boolean algebras 
with the FN also have the SFN. Naturally, 
Heindorf and Shapiro posed the question of whether 
the SFN is actually equivalent to the FN. Twenty years later,
there appears to have been no subsequent published work on the SFN.
The primary motivation of this work is to answer 
Heindorf and Shapiro's question.

\begin{theorem}\label{mainthm}
There is a boolean algebra of size $\al_2$ that
has the FN but not the SFN.
\end{theorem}

To prove the above, we require new characterizations of
the FN and SFN in terms of ``retrospective'' sequences of
countable elementary submodels, as we shall explain shortly.
We expect that the techniques we use here for separating
the FN and SFN will see wider application in the future,
and have stated many lemmas in anticipating generality. 

For additional background information about the classes
of boolean algebras defined by the SFN, FN, WFN, 
and projectivity, we refer the reader to~\cite{hs}.

\section{Retrospective sequences of elementary substructures}

Our proof of Theorem~\ref{mainthm} uses \lapps{\lm}s, 
which one can think of as a poor man's higher-gap morasses, 
available in ZFC. These sequences were introduced in~\cite{nthomog} 
as a more flexible version of Davies' 
trees of substructures~\cite{davies}.
Davies used such a tree to prove that the plane is a union
of countably many rotated graphs of functions;
Jackson and Mauldin~\cite{jackson} used such a tree to prove 
that there exists a subset of the plane intersecting every
isometric copy of $\integers^2$ at exactly one point.
The main application of \lapps{\lm}s in~\cite{nthomog}
was to prove that, for a class of topological spaces
that includes every compact group, every topological base 
of a space contains a base of the same space
which is upper finite with respect to inclusion.

\begin{definition}\
\itz
\item
Call a sequence of sets $(A_i)_{i\in I}$ \emph{retrospective}
if $I$ is an ordinal and, for all $i\in I$, the
the sequence $(A_j)_{j<i}$ is an element of $A_i$.
\item 
Given $\bighcard$ a regular uncountable cardinal and
$\lm$ a regular uncountable cardinal at most $\bighcard$,
call a set $M$ \emph{$\lm$-approximating} if
$\card{M}<\lm$, $M\cap\lm\in \lm$, and $M\elemsub\bigh$.
\item
Given $\bighcard$ and $\lm$ as above,
and $\eta$ an ordinal at most $\bighcard$,
a transfinite sequence $(M_i)_{i<\eta}$ is called
a \emph{\lapps{\lm}}\ if it is retrospective and
$M_i$ is $\lm$-approximating for all $i<\eta$.
\zti
\end{definition}

(In Definition 3.16 of~\cite{nthomog}, it was required 
of \lapps{\lm}s that also $\card{M_i}\subseteq M_i$ 
and $\lm\in M_i$. Here, we do not require 
$\card{M_i}\subseteq M_i$ because it is not needed 
for any applications (so far).
We do not require $\lm\in M_i$ for the
same reason, and because if $\lm\leq i<\eta$, 
then $\lm$ is definable in $M_i$
as $\sup_{j<i}\min(i\setminus M_j)$.)

The requirement $M_i\cap\lm\in\lm$ is succinct but perhaps
obscures its intended application, which is that for all
$A\in M_i$, if $\card{A}<\lm$, then $A\subseteq M_i$.
In particular, if $i<\lm$, then $\bigcup_{j<i}M_j\subseteq M_i$.
Also notice that the requirement
$M_i\cap\lm\in\lm$ is redundant if $\lm=\oml$.

\begin{lemma}\label{ubiquity}
Given regular uncountable cardinals $\lm\leq\bighcard$,
$A\in[\bigh]^{<\lm}$, $\eta<\bighcard$, and $(M_i)_{i<\eta}$
a \lapps{\lm}, there exists $M_\eta$ such that
$A\subseteq M_\eta$ and $(M_i)_{i<\eta+1}$ is a \lapps{\lm}.
\end{lemma}
\begin{proof}
Let $B=A\cup\{(M_i)_{i<\eta}\}$ and choose $M_\eta=\bigcup_{n<\om}N_n$
where $B\subseteq N_0$, $\card{N_n}<\lm$, $N_n\elemsub\bigh$,
and $N_n\cup\sup(\lm\cap N_n)\subseteq N_{n+1}$ for all $n$.
\end{proof}

\begin{lemma}\label{lappsorder}
Given $(M_i)_{i<\eta}$ as in the above definition and $\af,\beta<\eta$,
the following are equivalent.
 \enu
 \item\label{lo:ordinm} $\af\in\beta\cap M_\beta$
 \item\label{lo:minm} $M_\af\in M_\beta$
 \item\label{lo:msubm} $M_\af\subsetneq M_\beta$
 \une
\end{lemma}
\begin{proof}
Given \eqref{lo:ordinm}, we have \eqref{lo:minm} 
because $M_\af$ is definable from
$\af$ and $(M_\gamma)_{\gamma<\beta}$.
Given \eqref{lo:minm}, we have $M_\af\subseteq M_\beta$ because
$|M_\af|\in M_\beta\cap\lm\in\lm$; 
we also have $M_\af\in M_\beta\setminus M_\af$.
Given \eqref{lo:msubm}, we have $\af\not=\beta$; 
we also have $\af\in M_\beta$ because 
$\af$ is definable in $M_\af$ from $(M_\gamma)_{\gamma<\af}$; 
we also have $\af\leq\beta$ because otherwise 
$M_\af\subseteq\bigcup_{\gamma<\af}M_\gamma$, which is 
impossible because $\bigh\not\subseteq\bigcup_{\gamma<\af}M_\gamma$ 
and $\bigcup_{\gamma<\af}M_\gamma\in M_\af\elemsub\bigh$.
\end{proof}

\begin{lemma}\label{cover0}
If $S\in M_0$ and $(M_i)_{i<\card{S}}$ is a \lapps{\lm},
then $S\subseteq \bigcup_{i<\card{S}}M_i$.
\end{lemma}
\begin{proof}
Some $f\in M_0$ is a surjection from $\card{S}$ to $S$.
By Lemma~\ref{lappsorder}, $M_0\subsetneq M_\af$ for all $\af>0$, 
so $f(\af)\in M_\af$ for all $\af<\card{S}$.
\end{proof}

\begin{definition}
Given an ordinal $\af$ and an infinite cardinal $\lm$,
let the \emph{$\lm$-truncated cardinal normal form of $\af$}
denote the unique polynomial 
$$\om_{\beta_0}\gamma_0+\cdots+\om_{\beta_{m-1}}\gamma_{m-1}+\gamma_m$$
equal to $\af$ and satisfying $\om_{\beta_0}>\cdots>\om_{\beta_{m-1}}\geq\lm$, 
$\gamma_i\in[1,\om_{\beta_i}^+)$ for all $i<m$, and $\gamma_m<\lm$.
For each $i<m$, let $\cardterm{i}{\af}$ denote $\om_{\beta_i}\gamma_i$
and let $\cardhead{i+1}{\af}$ denote $\sum_{j<i+1}\cardterm{j}{\af}$; 
let $\cardhead{0}{\af}=0$ and $\cardhead{m+1}{\af}=\af$.
Let $\daleth(\af)$ denote $m$ if $\gamma_m=0$; $m+1$ if $\gamma_m>0$.
\end{definition}

Observe that $\cardhead{i}{\af}$ is $\{\af\}$-definable 
in $H(\card{\af}^+)$ for each $i\leq\daleth(\af)$. Hence, 
if $(M_\beta)_{\beta<\af}$ is a \lapps{\lm}, $\zeta+\eta\leq\af$, 
and $\cardhead{\daleth(\zeta)}{\zeta+\beta}=\zeta$ 
for all $\beta<\eta$, then, for each $\beta<\eta$,  
$(M_{\zeta+\gamma})_{\gamma<\beta}$ is definable in $M_{\zeta+\beta}$. 
Thus, such an $(M_{\zeta+\beta})_{\beta<\eta}$ is a \lapps{\lm}.
We will use this last observation to prove the fundamental lemma 
for \lapps{\lm}s, which is the existence
of a definable finite partition into directed segments.

\begin{lemma}\label{lappsfund}
Given a \lapps{\lm}\ $(M_\beta)_{\beta<\af}$,
the sets $\{M_\beta:\cardhead{i}{\af}\leq\beta<\cardhead{i+1}{\af}\}$ 
are directed with respect to inclusion for all $i<\daleth(\af)$.
\end{lemma}
\begin{proof}
A proof is implicit in the proof of Lemma 3.17 of~\cite{nthomog},
but we will provide a shorter explicit proof here.
Proceed by induction on $\af$.
If $\af\leq\lm$, then $\{M_\beta:\beta<\af\}$ is a chain.
If $\daleth(\af)\geq 2$, then each
$\{M_\beta:\cardhead{i}{\af}\leq\beta<\cardhead{i+1}{\af}\}$
is directed by our induction hypothesis applied to
$(M_{\cardhead{i}{\af}+\beta})_{\beta<\cardterm{i}{\af}}$.
So, suppose that $\af>\lm$ and $\daleth(\af)=1$.
If $\af=\sup\{\beta<\af:\daleth(\beta)=1\}$, then
$\{M_\gamma:\gamma<\af\}$ is directed because each
$\{M_\gamma:\gamma<\beta\}$ is directed by induction.
So, suppose that $\af=\ka(\gamma+1)$ where
$\ka$ is a cardinal and $1\leq\gamma<\ka^+$.
Set $\beta=\ka\gamma$ and $S=\{M_\delta:\delta<\beta\}$. 
By Lemma~\ref{cover0} applied to $S$ and $(M_{\beta+\delta})_{\delta<\ka}$,
we have $S\subset\bigcup_{\delta<\ka}M_{\beta+\delta}$. Hence,
by Lemma~\ref{lappsorder}, for every $\eps<\beta$ there
exists $\delta<\ka$ such that $M_\eps\subsetneq M_{\beta+\delta}$.
Therefore, $\{M_\delta:\delta<\af\}$ is directed because
its cofinal subset $\{M_{\beta+\delta}:\delta<\ka\}$ 
is directed by our inductive hypothesis applied to
$(M_{\beta+\delta})_{\delta<\ka}$.
\end{proof}

If $n<\om$ and $(M_\af)_{\af<\lm^{+n}}$ is a \lapps{\lm}, then,
since $\daleth(\af)\leq n+1$ for all $\af<\lm^{+n}$, we
can sometimes use $(M_\af)_{\af<\lm^{+n}}$ like a $(\lm,n)$-morass, 
in the weak sense that we can build a $\lm^{+n}$-sized object 
as a direct limit of small (that is, $(<\lm)$-sized) pieces 
while locally only having to fit together at most $n+1$ 
small direct limits of these small pieces.
Of course, we lack the additional coherence properties of
a $(\lm,n)$-morass, which require assumptions beyond ZFC.
However, the citations given at the beginning of this
section demonstrate that \lapps{\lm}\ are useful
even without such coherence. We will find them useful 
again in this paper. See also~\cite{dsoukup}
for very recent additional applications, noting that
there \lapps{\oml}s are called Davies sequences.

We finish this section with some additional lemmas
about \lapps{\lm}s that we will need later.

\begin{definition}
Given a \lapps{\lm}\ $(M_\beta)_{\beta<\eta}$, $\af\leq\eta$, and
$i<\daleth(\af)$, let
\itz
\item $I_i(\af)=[\cardhead{i}{\af},\cardhead{i+1}{\af})$;
\item $I'_i(\af)=I_i(\af)\cap M_\af$;
\item $\mc{I}_i(\af)=\{M_\beta: \beta\in I_i(\af)\}$;
\item $\mc{I}'_i(\af)=\{M_\beta: \beta\in I'_i(\af)\}$;
\item $M_{\af,i}=\bigcup\mc{I}_i(\af)$;
\item $M'_{\af,i}=M_{\af,i}\cap M_\af$.
\zti
\end{definition}

\begin{lemma}\label{lappscard}
If $(M_\af)_{\af<\eta}$ is a \lapps{\lm}, $i<\daleth(\eta)$, and 
$\cardterm{i}{\eta}\geq\lm$, then 
$\card{M_{\eta,i}}=\cardterm{i}{\eta}\subseteq M_{\eta,i}$.
\end{lemma}
\begin{proof}
Since $(M_{\cardhead{i}{\eta}+\af})_{\af<\cardterm{i}{\eta}}$ is a \lapps{\lm}, 
we have $\cardterm{i}{\eta}\subseteq M_{\eta,i}$. 
Since each $M_{\cardhead{i}{\eta}+\af}$ is smaller than $\lm$, 
we have $\cardterm{i}{\eta}=\card{M_{\eta,i}}$.
\end{proof}

\begin{lemma}\label{retro0}
If $(M_\af)_{\af<\eta}$ is a \lapps{\lm}\ and $S\in M_0$, 
then, for all $\af<\eta$, $S\in M_\af$ and
$S\in M_{\af,i}$ for all $i<\daleth(\af)$.
\end{lemma}
\begin{proof}
By Lemma~\ref{lappsorder}, $M_0\subseteq M_\af$ for all
$\af<\eta$. Hence, also $M_0\subseteq\bigcup_{\beta\in I_i(\af)}M_\beta$
for all $\af<\eta$ and $i<\daleth(\af)$.
\end{proof}

\begin{definition}
Because $\mc{I}_i(\af)$ and $\mc{I}'_i(\af)$ may not
be downward closed in $\{M_\beta:\beta<\eta\}$ with respect to 
inclusion, we also define
\itz
\item $J_i(\af)=\bigcup\{M_\beta\cap(\beta+1):\beta\in I_i(\af)\}$;
\item $J'_i(\af)=\bigcup\{M_\beta\cap(\beta+1):\beta\in I'_i(\af)\}$;
\item $\mc{J}_i(\af)=\{M_\beta: \beta\in J_i(\af)\}$;
\item $\mc{J}'_i(\af)=\{M_\beta: \beta\in J'_i(\af)\}$.
\zti
\end{definition}
By Lemma~\ref{lappsorder}, 
$\mc{J}_i(\af)$ and $\mc{J}'_i(\af)$ are downward closed in
$\{M_\beta:\beta<\eta\}$ with respect to inclusion.
Also observe that, by elementarity and retrospectiveness,
\itz
\item $\mc{I}'_i(\af)=\mc{I}_i(\af)\cap M_\af$;
\item $J'_i(\af)=J_i(\af)\cap M_\af$;
\item $\mc{J}'_i(\af)=\mc{J}_i(\af)\cap M_\af$.
\zti

\begin{lemma}\label{lappsref}
Given a \lapps{\lm}\ $(M_\beta)_{\beta<\af+1}$ and $i<\daleth(\af)$,
$\mc{I}_i(\af)$ and $\mc{I}'_i(\af)$ are directed 
with respect to inclusion 
with respective unions $M_{\af,i}$ and $M'_{\af,i}$.
Moreover, $\mc{I}_i(\af)$ is cofinal in $\mc{J}_i(\af)$
and $\mc{I}'_i(\af)$ is cofinal in $\mc{J}'_i(\af)$.
\end{lemma}
\begin{proof}
By Lemma~\ref{lappsfund}, $\mc{I}_i(\af)$ is directed;
by definition, its union is $M_{\af,i}$.
Since $(M_\beta:\beta\in I_i(\af))\in M_\af\elemsub\bigh$, 
the set $\mc{I}_i(\af)\cap M_\af$ is directed with union 
$M_{\af,i}\cap M_\af$.
Having thus proved the first sentence of the lemma,
the second sentence immediately follows from Lemma~\ref{lappsorder}.
\end{proof}

\begin{lemma}\label{strata}
Given a \lapps{\lm}\ $(M_\beta)_{\beta<\af+1}$, we have 
\itz
\item 
$\bigcup_{\beta<\af}M_\beta=\bigcup_{i<\daleth(\af)}M_{\af,i}$,
\item 
$M_\af\cap\bigcup_{\beta<\af}M_\beta=\bigcup_{i<\daleth(\af)}M'_{\af,i}$, 
and 
\item 
$M_{\af,i},M'_{\af,i}\elemsub\bigh$ for all $i<\daleth(\af)$.
\zti
\end{lemma}
\begin{proof}
Clearly, $\af=\bigcup_{i<\daleth(\af)}I_i(\af)$;
the two equations of the lemma immediately follow.
By Lemmas~\ref{lappsfund} and \ref{lappsref}, 
each $M_{\af,i}$ and each $M'_{\af,i}$ is a directed union
of elementary substructures of $\bigh$, so 
$M_{\af,i},M'_{\af,i}\elemsub\bigh$.
\end{proof}

\begin{lemma}\label{lappsindep}
Given a \lapps{\lm}\ $(M_\beta)_{\beta<\af+1}$ and $i<\daleth(\af)$,
we have $I'_i(\af)\not\subseteq\bigcup_{j\not=i}J_j(\af)$.
\end{lemma}
\begin{proof}
Since $\lm$ is regular and 
$|I_j(\af)|>|I_k(\af)|$ for all $j<k<\daleth(\af)$,
we have $|I_i(\af)|>\sum_{i<j<\daleth(\af)}|J_j(\af)|$.
Let $\beta=\min\left(I_i(\af)\setminus
\bigcup_{i<j<\daleth(\af)}J_j(\af)\right)$,
which is definable in $M_\af$ and thus in 
$I'_i(\af)\setminus\bigcup_{i<j<\daleth(\af)}J_j(\af)$.
Since $\beta\geq\cardhead{i}{\af}$, we also have
$\beta\not\in\bigcup_{j<i}J_j(\af)$. 
\end{proof}

\begin{definition}
Given a \lapps{\lm}\ $(M_\beta)_{\beta<\af+1}$ and 
nonempty $s\subseteq\daleth(\af)$, let
\itz
\item $K_s(\af)=\bigcap_{i\in s}J_i(\af)$;
\item $K'_s(\af)=\bigcap_{i\in s}J'_i(\af)$;
\item $\mc{K}_s(\af)=\bigcap_{i\in s}\mc{J}_i(\af)$;
\item $\mc{K}'_s(\af)=\bigcap_{i\in s}\mc{J}'_i(\af)$.
\zti
\end{definition}
Observe that, by elementarity and retrospectiveness,
\itz
\item $K'_s(\af)=K_s(\af)\cap M_\af$;
\item $\mc{K}'_s(\af)=\mc{K}_s(\af)\cap M_\af$;
\item $\mc{K}_s(\af)=\{M_\beta:\beta\in K_s(\af)\}$;
\item $\mc{K}'_s(\af)=\{M_\beta:\beta\in K'_s(\af)\}$.
\zti

\begin{lemma}\label{lappsint}
Given a \lapps{\lm}\ $(M_\beta)_{\beta<\af+1}$ and 
nonempty $s\subseteq\daleth(\af)$,
the sets $\mc{K}_s(\af)$ and $\mc{K}'_s(\af)$
are directed with respect to inclusion.
\end{lemma}
\begin{proof}
Since $\mc{K}_s(\af)\in M_\af\elemsub\bigh$, 
\istst\ $\mc{K}_s(\af)$ is directed.
Proceed by induction on $\card{s}$.
Case $\card{s}=1$ follows from Lemmas \ref{lappsfund}\ and \ref{lappsref}.
Assuming $\card{s}>1$, let $i=\max(s)$ and $t=s\setminus\{i\}$.
Suppose that $\beta,\gamma\in K_s(\af)$.
Since $\beta,\gamma\in J_i(\af)$, we have
$M_\beta,M_\gamma\subseteq M_\delta$ for some $\delta\in I_i(\af)$.
By definition, $K_t(\af)<I_i(\af)$; hence, $\beta,\gamma<\delta$; 
hence, $M_\beta,M_\gamma\in M_\delta$ by Lemma~\ref{lappsorder}.
Since $\delta\in I_i(\af)$, we have $I_j(\delta)=I_j(\af)$ for
all $j<i$. Therefore, $M_\delta$ knows that $\mc{K}_t(\af)$ 
is directed and that $M_\beta,M_\gamma\in\mc{K}_t(\af)$.
Hence, there exists $M_\eps\in M_\delta\cap \mc{K}_t(\af)$ such that
$M_\beta,M_\gamma\subseteq M_\eps$. By Lemma~\ref{lappsorder}, 
$\eps\in M_\delta\cap\delta$; hence, $\eps\in J_i(\af)$.
Thus, $M_\beta$ and $M_\gamma$ have a common superset $M_\eps$ 
in $\mc{K}_t(\af)\cap\mc{J}_i(\af)$, as desired.
\end{proof}

\begin{definition}
Given a \lapps{\lm}\ $(M_\beta)_{\beta<\eta}$ and $x\in\bigcup_{\beta<\eta}M_\beta$,
let the \emph{$M$-rank of $x$}, written $\rho(x,M)$ or just $\rho(x)$, 
denote the least $\af<\eta$ such that $x\in M_\af$.
\end{definition}

\begin{lemma}\label{rhodef}
Given a \lapps{\lm}\ $(M_\beta)_{\beta<\af+1}$ and $x\in M_\af$,
we have $M_{\rho(x)}\subseteq M_\af$.
\end{lemma}
\begin{proof}
Supposing $\rho(x)<\af$, we have  $M_{\rho(x)}$ definable in $M_\af$
from $x$ and $(M_\beta)_{\beta<\af}$. By Lemma~\ref{lappsorder},
we then have $M_{\rho(x)}\subsetneq M_\af$.
\end{proof}

\begin{lemma}\label{strataint}
For every \lapps{\lm}\ $(M_\beta)_{\beta<\eta}$ and 
$\vn\not=E\subseteq\eta$, there exists $D\subseteq\eta$
such that $\bigcap_{\af\in E}M_\af=\bigcup_{\af\in D}M_\af$
and $\{M_\af:\af\in D\}$ is directed.
\end{lemma}
\begin{proof}
Let $N=\bigcap_{\af\in E}M_\af$. By Lemma~\ref{rhodef}, 
$N=\bigcup\{M_\af:\af<\eta\text{ and }M_\af\subseteq N\}$. 
Suppose that $\af,\beta<\eta$ and $M_\af,M_\beta\subsetneq N$.
It suffices to find $\gamma<\eta$ such that 
$M_\af\cup M_\beta\subseteq M_\gamma\subseteq N$.
First, note that $M_\af,M_\beta\in M_i$ for all $i\in E$
by Lemma~\ref{lappsorder}. Since $E$ is nonempty,
we may define $\gamma=\rho(\{M_\af,M_\beta\})$.
We then have $M_\af,M_\beta\subsetneq M_\gamma$
(again by Lemma~\ref{lappsorder}). 
Fix $i\in E$; \istst\ $M_\gamma\subseteq M_i$.
By definition of $\rho$, $\gamma\leq i$. 
If $\gamma<i$, then $M_\gamma\subsetneq M_i$ because
$M_\gamma$ is definable in $M_i$.
\end{proof}

\begin{lemma}\label{stratadir}
For every \lapps{\lm}\ $(M_\af)_{\af<\eta}$ and $E\subseteq\eta$, 
if $\{M_\af:\af\in E\}$ is directed, then there exists
$i<\daleth(\eta)$ such that $E\subseteq J_i(\eta)$.
\end{lemma}
\begin{proof}
Let $\mc{E}=\{M_\af:\af\in E\}$ and 
$\mc{E}_i=\{M_\af:\af\in E\cap I_i(\eta)\}$
for each $i<\daleth(\eta)$.
Since $\mc{E}$ is directed and $\{\mc{E}_i:i<\daleth(\eta)\}$ 
is a finite partition of $\mc{E}$, there must exist $i$ such that
$\mc{E}_i$ is cofinal in $\mc{E}$.
By Lemma~\ref{lappsorder}, $E\subseteq J_i(\eta)$ for any such $i$.
\end{proof}

\section{Retrospective characterizations of the FN and SFN}

\begin{lemma}\label{projrestrict}
Given a poset $C$ and $A,B\subseteq C$ such that $A\commute B$
and $A\cap B$ is a relatively complete suborder of $A$, 
the functions $\pi_+^B\restrict A$ and $\pi_-^B\restrict A$
respectively equal 
$\pi_+^{A\cap B}\restrict A$ and $\pi_-^{A\cap B}\restrict A$.
\end{lemma}
\begin{proof}
Given $a\in A$ and $b\in B$ such that $a\leq b$, we
have some $c\in [a,b]\cap A\cap B$; hence,
$\pi_+^{A\cap B}(a)\leq c\leq b$; hence,
$\pi_+^B(a)$ exists and equals $\pi_+^{A\cap B}(a)$.
Likewise, $\pi_-^B(a)$ exists and equals $\pi_-^{A\cap B}(a)$.
\end{proof}

\begin{proposition}\label{communion}
If $C$ is a poset, $\mc{A},\mc{B}\subseteq\mc{P}(C)$,
and $A\commute B$ for all $(A,B)\in\mc{A}\times\mc{B}$,
then $\bigcup\mc{A}\commute\bigcup\mc{B}$.
\end{proposition}

\begin{proposition}\label{rcdown}
Given a poset $C$ and $A\subseteq B\subseteq C$,
if $A$ is a relatively complete suborder of $C$, 
then $A$ is relatively complete suborder of $B$.
\end{proposition}

\begin{definition}\
Given a boolean algebra $A$, a \lapps{\lm}\ $(M_\beta)_{\beta<\eta}$,
$x\in A\cap\bigcup_{\beta<\eta}M_\beta$, and $i<\daleth(\rho(x))$,
let $\pi_+^i(x,M)$ or just $\pi_+^i(x)$ denote 
$\pi_+^{M_{\rho(x),i}}(x)$ if it exists;
likewise let $\pi_-^i(x,M)$ or just $\pi_-^i(x)$ denote
$\pi_-^{M_{\rho(x),i}}(x)$ if it exists.
\end{definition}

\begin{theorem}\label{fnchar}
Let $A$ be a boolean algebra.
The following are equivalent.
\enu
\item\label{fn:fn} 
$A$ has the FN.
\item\label{fn:allpi}
For every \lapps{\oml}\ $(M_\af)_{\af<\card{A}}$ with $A\in M_0$,
for every $x\in A$, $\rho(x)$ exists and, for every $i<\daleth(\rho(x))$,
$\pi_+^i(x)$ and $\pi_-^i(x)$ exist.
\item\label{fn:somepi}
There exists a \lapps{\oml}\ $(M_\af)_{\af<\card{A}}$ such that,
for every $x\in A$, $\rho(x)$ exists and, for every $i<\daleth(\rho(x))$,
$\pi_+^i(x)$ and $\pi_-^i(x)$ exist.
\item\label{fn:allbigcomm}
For every \lapps{\oml}\ $(M_\af)_{\af<\card{A}}$ with $A\in M_0$,
for every $\af<\card{A}$, and for every $i<\daleth(\af)$, 
we have $A\cap M_\af\commute A\cap M_{\af,i}$ and
$A\cap M'_{\af,i}\subrc A\cap M_\af$.
\item\label{fn:somebigcomm}
There exists a \lapps{\oml}\ $(M_\af)_{\af<\card{A}}$ such that
$A\subseteq\bigcup_{\af<\card{A}}M_\af$ and, for all $\af<\card{A}$ and
$i<\daleth(\af)$, we have $A\cap M_\af\subalg A$,
$A\cap M_\af\commute A\cap M_{\af,i}$, and
$A\cap M'_{\af,i}\subrc A\cap M_\af$.
\item\label{fn:allsmallcomm}
For every \lapps{\oml}\ $(M_\af)_{\af<\card{A}}$ with $A\in M_0$,
for every $\af,\beta<\card{A}$, and for every $i<\daleth(\af)$, 
we have $A\cap M_\af\commute A\cap M_\beta$ and
$A\cap M'_{\af,i}\subrc A\cap M_\af$.
\item\label{fn:somesmallcomm}
There exists a \lapps{\oml}\ $(M_\af)_{\af<\card{A}}$ such that
$A\subseteq\bigcup_{\af<\card{A}}M_\af$ and, for all $\af,\beta<\card{A}$ and
$i<\daleth(\af)$, we have $A\cap M_\af\subalg A$,
$A\cap M_\af\commute A\cap M_\beta$, and
$A\cap M'_{\af,i}\subrc A\cap M_\af$.
\une
\end{theorem}
\begin{proof}
\eqref{fn:fn}$\Rightarrow$\eqref{fn:allpi}.
Fix $M$ as in the hypothesis of \eqref{fn:allpi}.
For each $x\in A$, $\rho(x)$ exists by Lemma~\ref{cover0}.
Each $M_{\rho(x),i}$ is an elementary substructure of 
$\bigh$ by Lemma~\ref{strata}. Also,
$A\in M_{\rho(x),i}$  by Lemma~\ref{retro0}. 
Hence, $A\cap M_{\rho(x),i}\subrc A$ by Theorem~\ref{fuchinochar}.
Hence, $\pi_+^i(x)$ and $\pi_-^i(x)$ exist.

\eqref{fn:fn}$\Rightarrow$\eqref{fn:allsmallcomm}.
Fix $M$ as in the hypothesis of \eqref{fn:allsmallcomm}.
By Lemma~\ref{retro0}, we have $A\in M'_{\af,i}$ 
for all $\af<\card{A}$ and $i<\daleth(\af)$.
Hence, by Lemma~\ref{strata} and Theorem~\ref{fuchinochar},
we have $A\cap M'_{\af,i}\subrc A$.
Hence, $A\cap M'_{\af,i}\subrc A\cap M_\af$
by Proposition~\ref{rcdown}. Finally,
given $\af,\beta<\card{A}$, choose an FN map $f\in M_0$.
By Lemma~\ref{retro0}, $M_\af$ and $M_\beta$ are $f$-closed;
hence, $A\cap M_\af\commute A\cap M_\beta$.

\eqref{fn:allpi}$\Rightarrow$\eqref{fn:somepi}, 
\eqref{fn:allbigcomm}$\Rightarrow$\eqref{fn:somebigcomm}, 
and \eqref{fn:allsmallcomm}$\Rightarrow$\eqref{fn:somesmallcomm}. 
Choose $\bighcard$ large enough that $A\in\bigh$.
By Lemma~\ref{ubiquity}, there is a \lapps{\oml} $(M_\af)_{\af<\card{A}}$ 
with $A\in M_0$. By Lemma~\ref{retro0}, we have
$A\cap M_\af\subalg A$ for all $\af<\card{A}$.
By Lemma~\ref{cover0}, we have $A\subseteq\bigcup_{\af<\card{A}}M_\af$.

\eqref{fn:allsmallcomm}$\Rightarrow$\eqref{fn:allbigcomm}
and \eqref{fn:somesmallcomm}$\Rightarrow$\eqref{fn:somebigcomm}.
Apply Proposition~\ref{communion}.

\eqref{fn:somebigcomm}$\Rightarrow$\eqref{fn:somepi}. 
By Lemma~\ref{projrestrict}, $\pi_+^{M_{\af,i}}$ and $\pi_-^{M_{\af,i}}$
are well-defined on all of $A\cap M_\af$.

\eqref{fn:somepi}$\Rightarrow$\eqref{fn:fn}. 
This is implicit in the author's proof of Corollary~3.4 of~\cite{mlkfn},
but we include a proof here for completeness.
For each $\af<\card{A}$, choose a well-ordering
$\sqsubseteq_\af$ of $\{x\in A:\rho(x)=\af\}$ with length at most $\om$. 
Set ${\sqsubseteq}=\bigcup_{\af<\card{A}}{\sqsubseteq_\af}$.
Recursively define $f\colon A\rightarrow[A]^{<\alo}$ by
$$f(x)=\{y: y\sqsubseteq x\}\cup\left(\bigcup_{i<\daleth(\rho(x))}
(f(\pi_+^i(x))\cup f(\pi_-^i(x)))\right).$$
Suppose $x\leq_A y$. We verify that $S=[x,y]\cap f(x)\cap f(y)$ 
is nonempty by induction on $\max\{\rho(x),\rho(y)\}$.
If $\rho(x)=\rho(y)$, then
$x\sqsubseteq y$, in which case $x\in S$,
or $y\sqsubseteq x$, in which case $y\in S$.
If $\rho(x)<\rho(y)$, then $x\in M_{\rho(y),i}$ for some $i$,
in which case $[x,\pi_-^i(y)]\cap f(x)\cap f(\pi_-^i(y))$ 
is a nonempty subset of $S$.
If $\rho(y)<\rho(x)$, then $y\in M_{\rho(x),i}$ for some $i$,
in which case $[\pi_+^i(x),y]\cap f(\pi_+^i(x))\cap f(y)$
is a nonempty subset of $S$.
\end{proof}

\begin{lemma}\label{commproj}
Given boolean algebras $A\subalg C$ and $B\subrc C$, 
the following are equivalent.
\enu
\item\label{cp:comm} $A\commute B$.
\item\label{cp:up} $\pi_+^B[A]\subseteq A$.
\item\label{cp:down} $\pi_-^B[A]\subseteq A$.
\une
\end{lemma}
\begin{proof}
\eqref{cp:comm}$\Rightarrow$\eqref{cp:up}.
Given $a\in A$, we have $a\leq \pi_+^B(a)\in B$, so
there exists $b\in [a,\pi_+^B(a)]\cap A\cap B$. However, by
definition of $\pi_+^B$, we must have have $b=\pi_+^B(a)$.

\eqref{cp:up}$\Rightarrow$\eqref{cp:comm}.
Given $a\in A$ and $b\in B$, if $a\leq b$, then
$\pi_+^B(a)\in[a,b]\cap A\cap B$; if $b\leq a$, then
$\pi_-^B(a)\in[b,a]\cap B$ and $\pi_-^B(a)=-\pi_+^B(-a)\in A$.

\eqref{cp:up}$\Leftrightarrow$\eqref{cp:down}.
$\pi_-^B(\bullet)=-\pi_+^B(-\bullet)$ and
$\pi_+^B(\bullet)=-\pi_-^B(-\bullet)$.
\end{proof}

\begin{definition}\
\itz
\item 
A \emph{partial algebra} is a pair of the form $(U,\mc{F})$
where $U$ is a set (called the universe of $(U,\mc{F})$) 
and $\mc{F}$ is a set of functions such that, for each $f\in\mc{F}$,
there exists $n<\om$ such that $\dom(f)\subseteq A^n$.
If every $\dom(f)$ is of the form $A^n$, then
we say that $(U,\mc{F})$ is an \emph{algebra}.
\item 
A partial algebra $(U,\mc{F})$ is a \emph{subalgebra} of 
a partial algebra $(V,\mc{G})$ if $U\subseteq V$,
$\mc{F}=\{g\restrict U^{<\om}:g\in\mc{G}\}$, 
and $\bigcup_{g\in\mc{G}}g[U^{<\om}]\subseteq U$.
\item 
A partial algebra $(U,\mc{F})$ is \emph{locally finite} if,
for every finite $A\subseteq U$, there exists a finite $B\subseteq U$
such that $A\subseteq B$ and $(B,\{f\restrict B:f\in\mc{F}\})$ 
is a subalgebra of $(U,\mc{F})$.
\item 
We say that a partial algebra $(V,\mc{G})$ \emph{expands} 
a partial algebra $(U,\mc{F})$ if $U=V$ and $\mc{F}\subseteq\mc{G}$. 
\item
Given a \lapps{\lm} $(M_i)_{i<\eta}$ and a boolean algebra $A$, 
let the \emph{$M$-expansion of $A$}, 
written $A[M]$, denote the expansion of $A$ resulting
from adding the functions in the set
$\bigcup_{i<\om}\{\pi_+^i,\pi_-^i\}$.
\zti
\end{definition}

\begin{theorem}\label{mlocfin}
If $A$ is a boolean algebra with the SFN, 
$(M_\af)_{\af<\card{A}}$ is a \lapps{\oml},
and $A\in M_0$, then $A[M]$ is locally finite.
\end{theorem}
\begin{proof}
Let $\mc{C}\in M_0$ be a pairwise commuting cofinal family of
finite subalgebras of $A$.
\Istst, for every $F\in\mc{C}$,
(an expansion of) $F$ is a subalgebra of $A[M]$.  
Let $x\in F\in\mc{C}$, $\af=\rho(x)$, and $i<\daleth(\af)$. 
By Theorem~\ref{fnchar}, $\pi_\pm^i(x)$ are well-defined. 
By Lemmas \ref{retro0} and \ref{strata}, $\mc{C}\in M_{\af,i}\elemsub\bigh$,
so we may choose $G\in M_{\af,i}\cap\mc{C}$ such that $\pi_\pm^i(x)\in G$. 
Since $G\in M_{\af,i}$ and $G$ is finite, $G\subseteq M_{\af,i}$;
hence, $\pi_\pm^i(x)\in G$ implies $\pi_\pm^i(x)=\pi_\pm^G(x)$.
By Lemma~\ref{commproj}, $\pi_\pm^G(x)\in F$.
\end{proof}

The proof of Theorem~\ref{mlocfin} implicitly shows much more.
Indeed, we can expand $A[M]$ by adding every function of the
form $\pi_\pm^N$ where $\mc{C}\in N\elemsub\bigh$, yet still obtain
a locally finite partial algebra. However, local finiteness of
$A[M]$ is strong enough for our purposes. As we shall 
show in Section~\ref{bigproof}, it is strictly stronger than the FN.

\begin{question}
If $A$ is a boolean algebra with the FN, $(M_\af)_{\af<\card{A}}$ is a \lapps{\oml},
$A\in M_0$, and $A[M]$ is locally finite, then does $A$ have the SFN?
\end{question}

We do not know the answer to the above question.
However, we will point out that if we broaden our consideration
to arbitrary expansions of boolean algebras, then 
characterizations of the FN and SFN are apparently easier to obtain.

\begin{definition}
Call an FN map $f$ \emph{transitive} if $f(y)\subseteq f(x)$ for
all $x\in\dom(f)$ and $y\in f(x)$.
\end{definition}

\begin{lemma}\label{transfn}
If $A$ has the FN, then $A$ has a transitive FN map.
\end{lemma}
\begin{proof}
Construct an FN map $f$ as in the proof of
\eqref{fn:allbigcomm}$\Rightarrow$\eqref{fn:fn}\ in Theorem~\ref{fnchar},
except use the following recursive definition of $f$:
$$f(x)=\{x\}\cup\left(\bigcup_{y\sqsubset x}f(y)\right)\cup
\left(\bigcup_{i<\daleth(\rho(x))}
(f(\pi_+^i(x))\cup f(\pi_-^i(x)))\right).$$
The proof of 
\eqref{fn:allbigcomm}$\Rightarrow$\eqref{fn:fn}\ in Theorem~\ref{fnchar}
still works {\it verbatim}, but now $f$ is also transitive.
\end{proof}

\begin{lemma}\label{commint}
If $A$ is a boolean algebra, $\mc{B}$ is a pairwise commuting family 
of relatively complete subalgebras of $A$, and $B_0,B_1\in\mc{B}$,
then $\mc{B}\cup\{B_0\cap B_1\}$ is pairwise commuting.
\end{lemma}
\begin{proof}
Let $C\in\mc{B}$. Suppose that $x\in B_0\cap B_1$, $y\in C$, and $x\leq y$.
By symmetry, \istst\ $[x,y]\cap B_0\cap B_1\cap C$ is nonempty.
Let $z=\pi_+^{C}(x)$, which is in $[x,y]\cap C$.
By Lemma~\ref{commproj}, $z$ is also in each of $B_0$ and $B_1$.
\end{proof}

\begin{definition}\
\itz
\item 
Given a partial algebra $B$, call a subalgebra $C$ of $B$
\emph{cyclic} if, for some $x\in C$, $C$ is the smallest 
subalgebra of $B$ that contains $\{x\}$.
\item
Given a boolean algebra $A$, we say that a partial algebra $B$
is \emph{strongly $A$-commuting} if $B$ has the same
universe as $A$, $B$ is locally finite, and, 
for all cyclic subalgebras $F$ and $G$ of $B$,
we have $F\commute G$ as suborders of $(A,\leq_A)$.
\zti
\end{definition}

\begin{theorem}\label{atomcomm}
Given a boolean algebra $A$,
\itz
\item 
$A$ has the FN if and only if there is a strongly
$A$-commuting algebra;
\item $A$ has the SFN if and only if there is a
strongly $A$-commuting algebra expanding $A$.
\zti
Moreover, the above is true if we replace
``strongly $A$-commuting algebra'' with
``strongly $A$-commuting partial algebra.''
\end{theorem}
\begin{proof}
If $f$ is a transitive FN map on $A$, then, letting
$(f_n(x))_{n<\om}$ surject from $\om$ to $f(x)$ for
each $x\in A$, the algebra $B$ with universe $A$
and set of functions $\{f_n:n<\om\}$ is strongly 
$A$-commuting.

Conversely,
given a strongly $A$-commuting partial algebra $B$,
construct an FN map $f$ by letting $f(x)$ be 
the minimal subalgebra of $B$ containing $\{x\}$,
for each $x\in A$.

Suppose $\mc{C}$ is a pairwise commuting cofinal 
family of finite subalgebras of $A$. By Lemma~\ref{commint},
we may assume that $\mc{C}$ is closed with respect to
pairwise intersection. For each $x\in A$, let $\mc{C}(x)$
denote the smallest element of $\mc{C}$ that contains $\{x\}$;
let $(f_n(x))_{n<\om}$ surject from $\om$ to $\mc{C}(x)$.
The expansion of $A$ formed by the adding the functions
from $\{f_n:n<\om\}$ is strongly $A$-commuting.

Conversely,
suppose that $B$ is a strongly $A$-commuting expansion of $A$.
Let $\mc{C}$ denote the set of finite subalgebras of $B$. 
Since $B$ is locally finite, $\mc{C}$ is a cofinal family
of finite subalgebras of $A$ (provided we identify each
$C\in\mc{C}$ with the subalgebra of $A$ that has the same 
universe). Moreover, by Proposition~\ref{communion},
$\mc{C}$ is pairwise commuting.
\end{proof}

Observe that adapting the proof of Lemma~\ref{transfn} 
to build a strongly $A$-commuting expansion of $A$
would require not only that $A[M]$ be locally finite, 
but also that $A[M]$ remain locally finite after adding 
a partial function that maps each $x\in A$ to its
immediate $\sqsubset$-predecessor, if one exists.

We shall need the next lemmas in Section~\ref{bigproof}.

\begin{lemma}\label{rcmeet}
If $A\subrc B$, $a\in A$, $b\in B$,
and $\pi_+^A(b)=x$, then 
$\pi_+^A(a\wedge b)=a\wedge x$.
\end{lemma}
\begin{proof}
\begin{align*}
a\wedge b\leq y\in A
&\Rightarrow b\leq y\vee-a\in A\\
&\Rightarrow b\leq x\leq y\vee -a\\
&\Rightarrow a\wedge b\leq a\wedge x\leq y\qedhere
\end{align*}
\end{proof}

\begin{definition}
Given a boolean algebra $A$, a \lapps{\lm}\ $(M_\af)_{\af<\eta}$, 
and $x\in A$, let $\sigma_+^\vn(x,M)=x$ and, 
for all $(t_0,\ldots,t_n)\in\om^{<\om}$, let 
$$\sigma_+^t(x,M)=(\pi_+^{t_n}\circ\pi_+^{t_{n-1}}\circ\cdots\circ\pi_+^{t_0})(x,M)$$
if the righthand side exists. Let $\varsigma_+(x,M)$ denote 
the set of all $\sigma_+^t(x,M)$ that exist. 
Likewise define $\sigma_-^t(x,M)$ and $\varsigma_-(x,M)$. 
We may suppress the dependence of $\varsigma_\pm$ and $\sigma_\pm^t$
on $M$ when convenient. 
\end{definition}

Observe that if $s$ is a strict initial segment of $t$
and $\sigma_+^t(x)$ exists, then $t_{\card{s}}<\daleth(\rho(\sigma_+^s(x)))$
and $\rho(\sigma_+^t(x))<\rho(\sigma_+^s(x))$.
Hence, by K\"onig's Lemma, $\varsigma_+(x)$ is finite; 
likewise, $\varsigma_-(x)$ is finite.

\begin{lemma}\label{rclimit}
Suppose we have a boolean algebra $A$
and a \lapps{\lm}\ $(M_\af)_{\af<\eta}$ such that
$\pi_\pm^i(x)$ exist for all $x\in A\cap\bigcup_{\af<\eta}M_\af$
and all $i<\daleth(\rho(x))$.
Then, for every $B\subalg A[M]$ where $B$ is of the form 
$A\cap\bigcup_{\af\in I}M_\af$, we have 
$\pi_+^B(x)=\bigwedge(B\cap\varsigma_+(x))$ and
$\pi_-^B(x)=\bigvee(B\cap\varsigma_-(x))$
for all $x\in A\cap\bigcup_{\af<\eta}M_\af$.
\end{lemma}
\begin{proof}
Suppose that $x\in A\cap\bigcup_{\af<\eta}M_\af$ and $x\leq y\in B$. 
By symmetry, \istst\ $x_+^B\leq y$ where
$x_+^B=\bigwedge(B\cap\varsigma_+(x))$.
Proceed by induction on $\max\{\rho(x),\rho(y)\}$.
Choose $\af\in I$ such that $y\in M_\af$.
By Lemma~\ref{rhodef}, $M_{\rho(y)}\subseteq M_\af$;
hence, $A\cap M_{\rho(y)}\subseteq B$.
If $\rho(x)=\rho(y)$, then $x\in B$, in which case $x_+^B=x\leq y$.
If $\rho(x)<\rho(y)$, then $x\in M_{\rho(y),i}$ for
some $i$, in which case $x\leq \pi_-^i(y)\in B$
and $\rho(\pi_-^i(y))<\rho(y)$; by induction,
$x_+^B\leq\pi_-^i(y)\leq y$.
If $\rho(y)<\rho(x)$, then $y\in M_{\rho(x),i}$ for some $i$,
in which case $z\leq y$ and $\rho(z)<\rho(x)$ where $z=\pi_+^i(x)$;
by induction, $z_+^B\leq y$; hence, $x_+^B\leq z_+^B\leq y$ 
because $\varsigma_+(z)\subseteq\varsigma_+(x)$.
\end{proof}

\begin{corollary}\label{rcnested}
Suppose we have a boolean algebra $A$
and a \lapps{\lm}\ $(M_\af)_{\af<\eta}$ such that
$\rho(x)$ and $\pi_\pm^i(x)$ exist 
for all $x\in A$ and all $i<\daleth(\rho(x))$.
Further suppose that $\nu$ is a regular uncountable cardinal,
$A,M\in P\elemsub H(\nu)$, and $\lm\cap P\in\lm+1$.
Then, for all $x\in A\setminus P$,
\begin{align*}
\pi_+^P(x)&=\bigwedge_{i<\daleth(\rho(x,M))}\pi_+^P(\pi_+^i(x,M))
\text{ and }\\
\pi_-^P(x)&=\bigvee_{i<\daleth(\rho(x,M))}\pi_-^P(\pi_-^i(x,M)).
\end{align*}
\end{corollary}

\section{Embeddings and colimits}

In this section, we collect some facts and specify
some notation concerning various colimits and
various classes of boolean embeddings.
We refer the reader to \cite{hs} and \cite{kopp}
for additional background information.

\begin{definition}\
\itz
\item Say that a sequence $(F_i)_{i\in I}$ of subalgebras of a
fixed boolean algebra is \emph{independent} if, for all
finite $J\subseteq I$ and all $x\in\prod_{j\in J}F_j$, if
$\bigwedge_{j\in J} x(j)=0$, then $x(j)=0$ for some $j\in J$.
\item Say that a sequence $(x_i)_{i\in I}$ of elements of a
fixed boolean algebra is \emph{independent} if 
$(\{x_i,-x_i,0,1\})_{i\in I}$ is independent.
\item Say that a boolean algebra $F$ is \emph{free} if it is generated
by the range of an independent sequence of elements of $F$.
\zti
\end{definition}

Fix once and for all a countably infinite free boolean algebra $\Fr_\om$ 
and an independent sequence $(\fr_n)_{n<\om}$ generating $\Fr_\om$ such
that $\Fr_\om$ and $\fr$ are definable in $H(\al_1)$ without parameters.
For each $S\subset\om$, let $\Fr_S$ denote the subalgebra of $\Fr_\om$
generated by $\{\fr_n:n\in S\}$.

\begin{definition}\
\itz
\item A \emph{boolean embedding} is an injective boolean homomorphism.
\item
Given two boolean algebras $C_0$ and $C_1$, a
\emph{coproduct} of $C_0$ and $C_1$ is a boolean algebra $C_0\oplus C_1$
with boolean embeddings 
$\oplus_0\colon C_0\rightarrow C_0\oplus C_1$ and
$\oplus_1\colon C_1\rightarrow C_0\oplus C_1$ such that
$\oplus_0[C_0]$ and $\oplus_1[C_1]$ are independent and
$\bigcup_{i<2}\oplus_i[C_i]$ generates $C_0\oplus C_1$.
These embeddings are called \emph{cofactor maps}.
\zti
\end{definition}

Coproducts always exist uniquely up to isomorphism.

\begin{definition}\
\itz
\item
Say that a boolean embedding $f\colon A\rightarrow B$ is
\emph{free} if there is an infinite free boolean algebra 
$F$ and a coproduct $A\oplus F$ such that $A\oplus F=B$
and $\oplus_0=f$.
\item 
If $\id_A$ is free embedding from $A$ to $B$, then
we say that $B$ is a \emph{free extension} of $A$
and write $A\subfree B$.
\item
Given boolean algebras $A\subalg B$, we say that $A$
\emph{splits} in $B$ if every ultrafilter of $A$ extends to
at least two ultrafilters of $B$. We say that
$A$ \emph{splits perfectly} in $B$ if, for all finite
$F\subseteq B$, the subalgebra generated by $A\cup F$
splits in $B$.
\item
We say that a boolean embedding $f\colon A\rightarrow B$
\emph{splits perfectly} if $f[A]$ splits perfectly in $B$.
\zti
\end{definition}

Every free embedding is relatively complete and splits
perfectly. Conversely, we have \emph{Sirota's Lemma}~\cite{sirota}, 
\ie\ if $f\colon A\rightarrow B$ is relatively complete and
perfectly splitting and $B$ is generated by $f[A]\cup C$
for some countable $C$, then $f$ is free. Also note that
the classes of relatively complete, perfectly splitting,
and free embeddings are each closed with respect to composition.
Moreover, for any composite boolean embedding $f\circ g$,
if $f$ is perfectly splitting, then so is $f\circ g$.

\begin{definition}\
\itz
\item A \emph{quotient} of a boolean algebra $A$ with
respect to an ideal $I$ is a boolean algebra $B$ with
a surjective homomorphism $f\colon A\rightarrow B$
with kernel $I$; $f$ is called the \emph{quotient map}
and $f(x)$ may be denoted by $x/I$.
\item
Given boolean embeddings $f\colon C\rightarrow A$ and
$g\colon C\rightarrow B$, define a \emph{pushout} 
$A \underset{C}{\pushout} B$ of $f$ and $g$ to be a quotient 
of a coproduct $A\oplus B$ with respect to the ideal $I$
generated by $\{\oplus_0(f(c))\wedge\oplus_1(g(-c)):c\in C\}$.
Thus, $A \underset{C}{\pushout} B$ is a colimit of the diagram
formed by $f$ and $g$.
\item
Given $f$ and $g$ as above such that also
$f=g=\id_{A\cap B}$, let $A\pushout B$ more
specifically denote a pushout of $f$ and $g$ such that
$\oplus_0(a)/I=a$ and $\oplus_1(b)/I=b$ for all $a\in A$
and $b\in B$.
\zti
\end{definition}

If $A$ and $B$ are boolean algebras such that their
intersection $A\cap B$ is also a common subalgebra,
then $A\pushout B$ exists as above and is characterized 
up to isomorphism as a boolean algebra $D$ in which
$A$ and $B$ are commuting subalgebras and
$A\cup B$ generates $D$.

\begin{lemma}\label{freepushout}
If $A=C_0\cap C_1$ and, for each $i<2$, 
we have $C_i=A\oplus B_i$ with cofactor
maps $\oplus_0=\id_A$ and $\oplus_1=\id_{B_i}$, then
$(A,B_0,B_1)$ is independent in $C_0\pushout C_1$
\end{lemma}
\begin{proof}
Suppose that $a\in A$, $b_0\in B_0$, $b_1\in B_1$,
and $a\wedge b_0\wedge b_1=0$ in $C_0\pushout C_1$. Then
$a\wedge b_0\leq \hat a\leq -b_1$ for some $\hat a\in A$
because $C_0\commute C_1$. Hence, $\hat a\wedge b_1=0$.
Since $(A,B_1)$ is independent, $\hat a=0$ or $b_1=0$.
If $\hat a=0$, then $a\wedge b_0=0$, in which
case $a=0$ or $b_0=0$ because $(A,B_0)$ is independent.
Thus, $a=0$, $b_0=0$, or $b_1=0$.
\end{proof}

\begin{definition}
Given a (nonempty) directed set $\mc{D}$ of boolean algebras
such that $A\subseteq B$ implies $A\subalg B$
for all $A,B\in\mc{D}$, endow the union $\bigcup\mc{D}$ 
of (the universes of) the algebras in $\mc{D}$
with the unique algebraic operations 
that make the inclusions
$(\id_A\colon A\rightarrow\bigcup\mc{D})_{A\in\mc{D}}$
a colimit of the inclusions 
$(\id_A\colon A\rightarrow B)_{\{A\subseteq B\}\subseteq\mc{D}}$,
namely, $\wedge_{\bigcup\mc{D}}=\bigcup_{A\in\mc{D}}\wedge_A$,
$\vee_{\bigcup\mc{D}}=\bigcup_{A\in\mc{D}}\vee_A$,
${-}_{\bigcup\mc{D}}=\bigcup_{A\in\mc{D}}{-}_A$,
$0_{\bigcup\mc{D}}=0_B$, and $1_{\bigcup\mc{D}}=1_B$
for some $B\in\mc{D}$.
\end{definition}

\section{Proof of main theorem}\label{bigproof}
By Theorem~\ref{mlocfin}, it is sufficient to construct
a boolean algebra $\Omega$ of size $\al_2$ such that $\Omega$ has
the FN, but $\Omega[N]$ is not locally finite
for some \lapps{\oml}\ $(N_\af)_{\af<\om_2}$ with $\Omega\in N_0$.
We will construct in parallel a sequence $(A_\af)_{\af<\om_2}$ of 
countable boolean algebras and a \lapps{\oml}\ $(M_\af)_{\af<\om_2}$ 
such that, for all $\af<\om_2$, we have
\begin{gather}
\label{mknowsa} 
A_\af\Subset M_\af\elemsub H(\al_2),\\
\label{newainnewm} 
A_\af\cap\bigcup_{\beta<\af}M_\beta 
=\bigcup_{i<\daleth(\af)}A'_{\af,i}\text{, and}\\
\label{oldaparts} 
\forall i<\daleth(\af)\,\ A'_{\af,i}\subalg A_\af
\end{gather}
where $A'_{\af,i}=\bigcup\{A_\beta:\beta\in J'_i(\af)\}$ and
$S\Subset T$ means that $S$ is a coinfinite subset of $T$.
Also define $A_{\af,i}=\bigcup\{A_\beta:\beta\in J_i(\af)\}$.

\begin{claim}
Given $A$ and $M$ as above, $\af,\beta<\om_2$, and $i<\daleth(\af)$,
we have that
\itz
\item $A_\af\cap A_\beta=A_\af\cap M_\beta$, 
\item $A_\af\subseteq A_\beta$ if $M_\af\subseteq M_\beta$, and
\item $A'_{\af,i}=A_\af\cap A_{\af,i}$.
\zti
\end{claim}
\begin{proof}
For the first subclaim,
we may assume we are not in the trivial case $\af=\beta$.
Since $A_\beta$ is countable, 
\eqref{mknowsa}\ implies $A_\beta\subseteq M_\beta$.
Therefore, $A_\af\cap A_\beta\subseteq A_\af\cap M_\beta$. 
To prove the converse inclusion, suppose $x\in A_\af\cap M_\beta$.
If $\beta<\af$, then, by \eqref{newainnewm}, 
$x\in A_\gamma$ for some $\gamma<\af$;
we may inductively assume $A_\gamma\cap M_\beta=A_\gamma\cap A_\beta$,
in which case $x\in A_\beta$.  If $\af<\beta$, then 
$\rho(x)\in M_\beta\cap I_j(\beta)\subseteq J'_j(\beta)$ 
for some  $j<\daleth(\beta)$; hence, 
by \eqref{newainnewm}\ and \eqref{oldaparts}, 
$x\in A_{\rho(x)}\subseteq A'_{\beta,j}\subseteq A_\beta$.
Thus, $A_\af\cap M_\beta\subseteq A_\af\cap A_\beta$.
Therefore, $A_\af\cap A_\beta=A_\af\cap M_\beta$.

We obtain the third subclaim
from the first subclaim and from \eqref{oldaparts}:
\begin{multline*}
A_\af\cap A_{\af,i}
=A_\af\cap\bigcup_{\beta\in J_i(\af)}A_\beta
=A_\af\cap\bigcup_{\beta\in J_i(\af)}M_\beta
=A_\af\cap M_\af\cap\bigcup_{\beta\in J_i(\af)}M_\beta\\
=A_\af\cap\bigcup_{\beta\in J'_i(\af)}M_\beta
=A_\af\cap\bigcup_{\beta\in J'_i(\af)}A_\beta
=A_\af\cap A'_{\af,i}
=A'_{\af,i}.
\end{multline*}
Also, the second subclaim holds because
if $\delta<\om_2$ and $M_\af\subseteq M_\delta$, 
then $A_\af\subseteq A_\delta$ because 
$A_\af\cap A_\delta=A_\af\cap M_\delta\supseteq A_\af\cap M_\af=A_\af$.
\end{proof}

By the above claim, since $\mc{I}_0(\om_2)$ is directed,
letting $\Omega=A_{\om_2,0}=\bigcup_{\af<\om_2}A_\af$, 
we obtain a boolean algebra of size at most $\al_2$ such that,
for all $\af<\om_2$, $\Omega\cap M_\af=A_\af\subalg \Omega$ and,
for all $i<\daleth(\af)$, $\Omega\cap M_{\af,i}=A_{\af,i}$ and 
$\Omega\cap M'_{\af,i}=A'_{\af,i}$.
Therefore, by Theorem~\ref{fnchar}, $\Omega$ will have the FN if
$A'_{\af,i}\subrc A_\af$ and $A_\af\commute A_\beta$ as suborders of $\Omega$ 
for all $\af,\beta<\om_2$ and $i<\daleth(\af)$. Therefore, 
$\Omega$ will have the FN if we have the following for all $\af<\om_2$. 
\begin{gather}
\label{newafreeext} 
\forall i<\daleth(\af)\,\ A'_{\af,i}\subfree A_\af.\\
\label{oldacomm}
\forall M_\beta,M_\gamma\in M_\af\,\  
A_\beta\commute A_\gamma\text{ as suborders of }A_\af.
\end{gather}
Note that \eqref{newafreeext}\ implies \eqref{oldaparts}.

At stage $0$, let $A_0=\Fr_\om\in M_0$.
At nonzero stages $\af<\om_2$, select a countable $M_\af\elemsub H(\al_2)$
such that $(A_\beta,M_\beta)_{\beta<\af}\in M_\af$.
If $\daleth(\af)=1$, then let $A_\af$ be a coproduct 
$A'_{\af,0}\oplus\Fr_\om$ such that $A'_{\af,0}\subalg A_\af\Subset M_\af$
and $A_\af\setminus A'_{\af,0}$ is disjoint from $\bigcup_{\beta<\af}M_\beta$. 
Clearly, \eqref{mknowsa}, \eqref{newainnewm}, and \eqref{newafreeext}\ 
are preserved.
Since $\mc{J}'_0(\af)$ is directed, \eqref{oldacomm}\ is preserved too.

Now suppose that $\daleth(\af)=2$. 
Let $A'_{\af,2}=\bigcap_{j<2}A'_{\af,j}$.
By Lemma~\ref{lappsint}, $A'_{\af,2}$ is a directed union
of common subalgebras of $A'_{\af,0}$ and $A'_{\af,1}$.

\begin{claim}\label{rcdir}
$A'_{\af,2}\subrc A'_{\af,i}$ for each $i<2$.
\end{claim}
\begin{proof}
Fix $i<2$. By Lemma~\ref{strataint},
for each $\beta\in J'_i(\af)$ and $j<\daleth(\beta)$, 
$A'_{\af,i}\cap A_{\beta,j}=\bigcup\{A_\gamma:\gamma\in U\}$ 
for some $U\subseteq J'_i(\af)$.  By the inductive assumption 
of \eqref{oldacomm}\ for all stages before $\af$
and by the directedness of $\{A_\delta:\delta\in J'_i(\af)\}$,
$A_\gamma\commute A_\beta$ for all $\gamma\in U$.
Hence, by Proposition~\ref{communion},
$A'_{\af,i}\cap A_{\beta,j}\commute A_\beta$.
By the inductive assumption of \eqref{newafreeext} 
for all stages before $\af$, we have 
$A'_{\af,i}\cap A_{\beta,j}\cap A_\beta=A'_{\beta,j}\subrc A_\beta$.
Hence, by Lemma~\ref{projrestrict}, 
$\pi_\pm^{A'_{\af,i}\cap A_{\beta,j}}(x)$ exist for all $x\in A_\beta$. 
Therefore, applying Lemma~\ref{rclimit}
with $A'_{\af,i}$ in place of $A$ and $A'_{\af,2}$ in place of $B$,
we have $A'_{\af,2}\subrc A'_{\af,i}$.
\end{proof}

\begin{claim}\label{freehorns}
$A'_{\af,2}\subfree A'_{\af,i}$ for each $i<2$.
\end{claim}
\begin{proof}
Fix $i<2$. By Sirota's Lemma and the previous claim,
\istst\ $A'_{\af,2}$ splits perfectly in $A'_{\af,i}$.
Let $U$ be an ultrafilter of $A'_{\af,2}$ and
$F$ be a finite subset of $A'_{\af,i}$.
Applying Lemma~\ref{lappsindep},
there exists $\beta\in I'_i(\af)\setminus J_{1-i}(\af)$.
Since $\mc{I}'_i(\af)$ is directed and $\mc{J}_{1-i}(\af)$ is
downward closed in $\{M_\beta:\beta<\af\}$,
we may choose $\beta\in I'_i(\af)\setminus J_{1-i}(\af)$
such that $F\subseteq A_\beta$.
By Lemmas \ref{strataint} and \ref{stratadir},
$A_\beta\cap A'_{\af,2}\subseteq A'_{\beta,j}$ for some $j<\daleth(\beta)$.
Let $B\subalg A_\beta$ denote the subalgebra generated by 
$A'_{\beta,j}\cup F$. Extend $U\cap A_\beta$ to an ultrafilter $V$ of $B$.
Since $A'_{\beta,j}\subfree A_\beta$ by 
\eqref{newafreeext}\ for stage $\beta$,
$B$ splits in $A_\beta$. So, choose $y\in A_\beta$ and ultrafilters 
$V_\pm$ of $A_\beta$ that respectively extend $V\cup\{\pm y\}$.
By \eqref{oldacomm}\ for stages before $\af$ 
and by Proposition~\ref{communion},
$A'_{\af,2}\commute A_\beta$ as suborders of $A'_{\af,i}$, 
so both of $U\cup V_\pm$ extend to ultrafilters of $A'_{\af,i}$.
Thus, $A'_{\af,2}$ splits perfectly in $A'_{\af,i}$.
\end{proof}

Choose $B_\af=A'_{\af,0}\pushout A'_{\af,1}$
such that $B_\af\Subset M_\af$ and $B_\af\setminus(A'_{\af,0}\cup A'_{\af,1})$
is disjoint from $\bigcup_{\beta<\af}M_\beta$. 

\begin{claim}\label{oldacommpres}
\eqref{oldacomm}\ will be preserved if $B_\af\subalg A_\af$.
\end{claim}
\begin{proof}
Suppose that $x_i\in A_{\beta_i}$ and $M_{\beta_i}\in M_\af$ for each $i<2$, 
and that $x_0\leq x_1$ in $A_\af$. 
It suffices to find $w\in[x_0,x_1]\cap M_{\beta_0}\cap M_{\beta_1}$.
Let $\beta_i\in J'_{j_i}$ for each $i<2$. We inductively assume 
that \eqref{oldacomm}\ holds for all stages before $\af$.
Therefore, since $\mc{J}'_0(\af)$ and $\mc{J}'_1(\af)$ are directed, 
if $j_0=j_1$, then $w$ as above exists. So, assume that $j_0\not=j_1$.
By symmetry, we may assume that $j_0=0$.
Assuming $B_\af\subalg A_\af$, we have 
$A'_{\af,0}\commute A'_{\af,1}$ as suborders of  $A_\af$. Hence, 
we may choose $y\in A'_{\af,0}\cap A'_{\af,1}\cap[x_0,x_1]$.
Choose $\gamma\in K'_2(\af)$ such that $y\in M_\gamma$.
We may choose $z_0\in[x_0,y]\cap A_{\beta_0}\cap A_\gamma$ and
$z_1\in[y,x_1]\cap A_{\beta_1}\cap A_\gamma$, again
because $\mc{J}'_0(\af)$ and $\mc{J}'_1(\af)$ are directed.
For each $i<2$, let $\delta_i=\rho(z_i,M\restrict\af)$; 
we then have $z_i\in M_{\delta_i}\subseteq M_{\beta_i}\cap M_\gamma$
by Lemma~\ref{rhodef}. 
Since $\mc{K}'_2(\af)$ is also directed, we may choose
$w\in[z_0,z_1]\cap A_{\delta_0}\cap A_{\delta_1}$.
Thus, $w\in[x_0,x_1]\cap A_{\beta_0}\cap A_{\beta_1}$.
\end{proof}

Given Claim~\ref{freehorns}, we may choose, for each $i<2$,
cofactor maps $\oplus_0=\id\colon A'_{\af,2}\rightarrow A'_{\af,i}$
and $\oplus_1=\zeta_{\af,i}\colon \Fr_\om\rightarrow A'_{\af,i}$. 
For each $i<2$, let $B_{\af,i}=\ran(\zeta_{\af,i})$ and 
$b_{\af,i}^n=\zeta_{\af,i}(\fr_n)$ for each $n<\om$.
Choose $C_\af=B_\af\oplus\Fr_2$ such that 
$C_\af\Subset M_\af$,
that $\oplus_0=\id_{B_\af}$, and that
$C_\af\setminus B_\af$ is disjoint from $\bigcup_{\beta<\af}M_\beta$.
For each $i<2$, let $h_{\af,i}=\eta_\af(\fr_i)$
where $\eta_\af$ is the cofactor map 
$\oplus_1\colon \Fr_2\rightarrow C_\af$.
Let $H_\af=\ran(\eta_\af)$.
By Lemma~\ref{freepushout}, $(A'_{\af,2},B_{\af,0},B_{\af,1})$
is independent; hence, so is $(A'_{\af,2},B_{\af,0},B_{\af,1},H_\af)$.
For each $n<\om$, let 
$e^n_{\af,0}=b_{\af,0}^n\wedge b_{\af,1}^n\wedge h_{\af,0}$ and
$e^n_{\af,1}=b_{\af,0}^{n+1}\wedge b_{\af,1}^n\wedge h_{\af,1}$;
let $I_\af$ be the ideal of $C_\af$ generated by 
$\{e^n_{\af,i}:(i,n)\in 2\times\om\}$. 

\begin{claim}
$I_\af\cap(B_\af\cup H_\af)=\{0\}$.
\end{claim}
\begin{proof}
Let $1\leq m<\om$, let $e=\bigvee_{(i,n)\in 2\times m}e_{\af,i}^n$,
and let $J$ be the ideal generated by $e$.
We will show that $J\cap(B_\af\cup H_\af)=\{0\}$.
First, if $0<h\in H_\af$, then $h\not\leq e$ 
because $e\wedge c=0<c\leq h$ where $c=h\wedge -b_{\af,1}^0$.
Second, if $0<b\in B_\af$, then $b\not\leq e$
because $e\wedge c=0<c\leq b$ where $c=b\wedge -h_{\af,0}\wedge -h_{\af,1}$.
\end{proof}

By the above claim, we may choose a quotient
$D_\af=C_\af/I_\af$ such that $x/I_\af=x$ for all $x\in B_\af\cup H_\af$.
Choose $D_\af$ such that also 
$D_\af\Subset M_\af$ and $D_\af\setminus(B_\af\cup H_\af)$ is disjoint 
from $\bigcup_{\beta<\af}M_\beta$.
Note that $B_\af\not\subrc D_\af$ because, for example,
$\pi_+^{B_\af}(h_\af)$ does not exist because
$h_{\af,0}\leq_{D_\af} \bigwedge_{n<m}-(b_{\af,0}^n\wedge b_{\af,1}^n)$
for all $m<\om$. However, we still have the following.

\begin{claim}\label{roundup}
For each $i<2$, $A'_{\af,i}\subrc D_\af$. 
In particular, $\pi_+^{A'_{\af,0}}$ and $\pi_+^{A'_{\af,1}}$ satisfy
\eqref{essproj}\ for all $n<\om$ and $x\in D_\af$ of the forms below.
\begin{equation}\label{essproj}
\begin{array}{|l|l|l|}
x& \pi_+^{A'_{\af,0}}(x) & \pi_+^{A'_{\af,1}}(x)\\\hline
b_{\af,0}^n\wedge h_{\af,0}& b_{\af,0}^n & -b_{\af,1}^n\\\hline
b_{\af,1}^n\wedge h_{\af,1}& -b_{\af,0}^{n+1} & b_{\af,1}^n\\\hline
\end{array}
\end{equation}
\end{claim}
\begin{proof}
For this proof our notation will suppress the dependence on $\af$.
Every nonzero element of $C$ is a finite nonempty join of 
elements of the form $x=a\wedge b_0\wedge b_1\wedge h$ where 
$a\in A'_2$, $h$ is of the form $\pm h_0\wedge\pm h_1$, 
and each $b_i$ is $\bigwedge_{n\in P_i}b_i^n\wedge\bigwedge_{n\in Q_i}-b_i^n$
where $P_i$ and $Q_i$ are each (possibly empty) finite subsets of $\om$.
(Our convention is that $\bigwedge\vn=1$.)
In general, $\pi_+^{A'_i}(y\vee z)=\pi_+^{A'_i}(y)\vee\pi_+^{A'_i}(z)$
and $\pi_-^{A'_i}(y)=-\pi_+^{A'_i}(-y)$ if the righthand sides exist. 
Moreover, by Lemma~\ref{rcmeet}, 
$$\pi_+^{A'_i}((a\wedge b_i\wedge b_{1-i}\wedge h)/I)=
a\wedge b_i\wedge\pi_+^{A'_i}((b_{1-i}\wedge h)/I)$$ if the righthand side exists.

Let $x=b_{1-i}\wedge_C h$ where $b_{1-i}$ and $h$ are as above.
We will show that each $\pi_+^{A'_i}(x/I)$ exists and equals $\tau_i(x)$ 
where $\tau_i(x)=\bigwedge_{n\in T_i}-b_i^n$ where
$T_i$ is as in \eqref{roundup:targets} below, 
which uses shift operator notation 
$S\shiftr=\{\beta+1:\beta\in S\}$ and $S\shiftl=\{\beta:\beta+1\in S\}$
for sets of ordinals.
\begin{equation}\label{roundup:targets}
\begin{array}{|rcr|l|l|l|l|}
      &h&       & T_0 & T_1\\\hline
-h_0&\wedge&-h_1& \vn & \vn\\\hline
 h_0&\wedge&-h_1& P_1 & P_0\\\hline
-h_0&\wedge& h_1& P_1\shiftr & P_0\shiftl\\\hline
 h_0&\wedge& h_1& P_1\cup(P_1\shiftr) & P_0\cup(P_0\shiftl)\\\hline
\end{array}
\end{equation}
In all cases, $x/I\leq\tau_i(x)$ follows directly 
from the definition of $I$. Moreover, \eqref{essproj}\ follows
from \eqref{roundup:targets}.

Henceforth working in $C$, suppose that 
$y\in A'_i$, $t\in I$, and $x\leq y\vee t$.
We will show that $\tau_i(x)\leq y\vee e$ for some $e\in I$.
Every element of $A'_i\setminus\{1\}$ is a nonempty finite meet 
of elements of the form $z\vee w_i$ where $z\in A'_2\setminus\{1\}$ and 
$w_i$ is $\bigvee_{n\in R_i}b_i^n\vee\bigvee_{n\in S_i}-b_i^n$
where $R_i$ and $S_i$ are each (possibly empty) finite subsets of $\om$
and $R_i\perp S_i$. (Our convention is that $\bigvee\vn=0$.)
Moreover, by~\eqref{roundup:targets},
$\tau_i(x)=1$ if and only if $h=-h_0\wedge-h_1$ or $P_{1-i}=\vn$.
Therefore, it is enough to assume that $y=z\vee w_i$ where $z$ and $w_i$ are as above
and prove that $h\leq h_0\vee h_1$,
that $\tau_i(x)\leq y\vee e$ for some $e\in I$,
and that $P_{1-i}\not=\vn$.

By assumption, we have $b_{1-i}\wedge h\leq z\vee w_i\vee t$;
by independence of $(A'_2,B_0,B_1,H)$, we have $b_{1-i}\wedge h\leq w_i\vee t$,
which implies $b_{1-i}\wedge h\leq w_i\vee (h_0\vee h_1)$;
by independence of $(B_0,B_1,H)$, we have $h\leq h_0\vee h_1$.
Choose $m<\om$ and $e=\bigvee_{\substack{j<2\\n<m}}e^n_i$
such that $m\supseteq S_i$ and $e\geq t$.
Let $g_{0,n}(0)=b_0^n$, $g_{1,n}(0)=b_0^{n+1}$, $g_{j,n}(1)=b_1^n$, 
and $g_{j,n}(2)=h_j$ for all $(j,n)\in 2\times\om$. Then 
$$b_{1-i}\wedge h\leq w_i\vee e
=w_i\vee \bigvee_{\substack{j<2\\n<m}}\ \bigwedge_{k<3}g_{j,n}(k)
=w_i\vee\bigwedge_{f\colon 2\times m\rightarrow 3}\ 
\bigvee_{\substack{j<2\\n<m}}g_{j,n}(f(j,n)).$$
Thus, $b_{1-i}\wedge h\leq w_i\vee e$ if and only if,
for all $f\colon 2\times m\rightarrow 3$,
\begin{equation}\label{roundup:distribute}
b_{1-i}\wedge h\leq w_i\vee\bigvee_{\substack{j<2\\n<m}}g_{j,n}(f(j,n)).
\end{equation}

Given any $f\colon 2\times m\rightarrow 3$, let, 
for each $(j,k)\in 2\times 3$, $E_{j,k}=\{n:f(j,n)=k\}$;
let $F_{1,0}=E_{1,0}\shiftr$ and $F_{1,1}=E_{1,1}$.
Independence of $(B_0,B_1,H)$, $(b_0^n)_{n<\om}$,
and $(b_1^n)_{n<\om}$ imply that \eqref{roundup:distribute} is
equivalent to the disjunction of
\enu
\item[(X1)] $P_{1-i}\not\perp E_{0,1-i}\cup F_{1,1-i}$;
\item[(X2)] $S_i\not\perp E_{0,i}\cup F_{1,i}$;
\item[(X3)] $h\leq h_0$ and $E_{0,2}\not=\vn$;
\item[(X4)] $h\leq h_1$ and $E_{1,2}\not=\vn$.
\une 
Similarly, we have $\tau_i(x)\leq w_i\vee e$ if and only if, for
all $f\colon 2\times m\rightarrow 3$, we have 
$S_i\not\perp E_{0,i}\cup F_{1,i}$ or $T_i\not\perp S_i$.
By~\eqref{roundup:targets},
$T_i\not\perp S_i$ if and only if $P_{1-i}\not\perp U_i$ 
where $U_i$ is as in \eqref{roundup:shifted} below.
\begin{equation}\label{roundup:shifted}
\begin{array}{|rcr|l|l|}
     &h&   & U_0 & U_1\\\hline
 h_0&\wedge&-h_1 & S_0 & S_1\\\hline
-h_0&\wedge& h_1 & S_0\shiftl & S_1\shiftr\\\hline
 h_0&\wedge& h_1 & S_0\cup(S_0\shiftl) & S_1\cup(S_1\shiftr)\\\hline
\end{array}
\end{equation}
By choosing $f$ according to \eqref{roundup:shave} below,
we ensure that (X2), (X3), and (X4) fail, and, therefore,
that (X1) holds.
\begin{equation}\label{roundup:shave}
\begin{array}{|rcr|l|l|l|l|l|l|l|l|l|}
      &h&       & i & E_{0,2}&E_{1,2} & E_{0,1-i} & E_{1,1-i} & m\setminus F_{1,i} & F_{1,1-i}\\\hline
h_0&\wedge&-h_1 & 0 & \vn&m          & S_0 & \vn & \vn & \vn\\\hline
h_0&\wedge&-h_1 & 1 & \vn&m          & S_1 & \vn & \vn & \vn\\\hline
-h_0&\wedge&h_1 & 0 &   m&\vn        & \vn & S_0\shiftl & S_0\cup\{0\} & S_0\shiftl\\\hline
-h_0&\wedge&h_1 & 1 &   m&\vn        & \vn & S_1 & S_1 & S_1\shiftr\\\hline
 h_0&\wedge&h_1 & 0 & \vn&\vn        & S_0 & S_0\shiftl & S_0\cup\{0\} & S_0\shiftl\\\hline
 h_0&\wedge&h_1 & 1 & \vn&\vn        & S_1 & S_1 & S_1 & S_1\shiftr\\\hline
\end{array}
\end{equation}
Comparing \eqref{roundup:shifted} with \eqref{roundup:shave},
we see that $E_{0,1-i}\cup F_{1,1-i}=U_i$ in all cases.
Therefore, $P_{1-i}\not\perp U_i$. 
Thus, $\tau_i(x)\leq y\vee e$ and $P_{1-i}\not=\vn$.
\end{proof}

Choose $A_\af=D_\af\oplus\Fr_\om$ such that 
$A_\af\Subset M_\af$,
that $\oplus_0=\id_{D_\af}$, and that
$A_\af\setminus D_\af$ is disjoint from $\bigcup_{\beta<\af}M_\beta$.
Thus, \eqref{mknowsa}\ is preserved.
By construction, $A_\af\cap\bigcup_{\beta<\af}M_\beta=\bigcup_{i<2}A'_{\af,i}$,
so \eqref{newainnewm}\ is preserved.
By Claim~\ref{oldacommpres}, \eqref{oldacomm}\ is preserved.
By Sirota's Lemma, since $A'_{\af,i}\subrc D_\af\subrc A_\af$ for all $i<2$
and $D_\af$ splits perfectly in $A_\af$, \eqref{newafreeext}\ is preserved.
Thus, our construction of $\Omega$ is complete and $\Omega$ has the FN.

Choose a \lapps{\oml}\ $(N_\af)_{\af<\om_2}$ with 
$N_\af\elemsub H(\al_3)$ for all $\af<\om_2$ and
$A,M\in N_0$ (which implies $\Omega\in N_0$).
We will show that $\Omega[N]$ is not locally finite.
Let $\delta=\oml+1$; let $\beta=\om_2\cap N_{\delta,0}$, which
is in $\om_2\cap N_{\delta,1}$ and has cofinality $\oml$
by Lemma~\ref{lappscard};
let $[\beta,\af)=[\beta,\beta+\oml)\cap N_{\delta,1}$.
Note that $\af\in N_\delta$ and $\cardhead{1}{\af}=\beta$.

\begin{claim}\label{reflect}
$\beta\cap N_{\delta,1}=\beta\cap M_{\af,1}$.
\end{claim}
\begin{proof}
By Lemma~\ref{cover0}, $\beta\subseteq\bigcup_{\theta<\oml}M_{\beta+\theta}$.
Since $\beta,M\in N_{\delta,1}\elemsub H(\al_3)$, we have
$$\beta\cap N_{\delta,1}=\beta\cap\bigcup
\{M_\theta:\theta\in[\beta,\beta+\oml)\cap N_{\delta,1}\}
=\beta\cap\bigcup_{\beta\leq\theta<\af}M_\theta=\beta\cap M_{\af,1}.\qedhere$$
\end{proof}

We will show that $\{h_{\af,0},h_{\af,1},b_{\af,0}^0\}$ 
generates an infinite subalgebra of $\Omega[N]$.
\Istst\ 
\itz
\item $\rho(b_{\af,i}^m\wedge h_{\af,i},N)=\delta$ for all $i<2$ and $m<\om$,
\item $\pi_+^0(b_{\af,1}^m\wedge h_{\af,1},N)=-b_{\af,0}^{m+1}$ for all $m<\om$, and
\item $\pi_+^1(b_{\af,0}^m\wedge h_{\af,0},N)=-b_{\af,1}^m$ for all $m<\om$.
\zti

First, $\rho(b_{\af,i}^m\wedge h_{\af,i},M)=\af$ by construction.
By definition of $\af$, we have $\af\in N_\delta\setminus N_{\delta,1}$;
since $\beta\leq\af$, we also have $\af\not\in N_{\delta,0}$.
Since $M,\rho(\bullet,M)\in N_0$ and $M_\af$ is countable, we conclude that
$b_{\af,i}^m\wedge h_{\af,i}\in N_\delta\setminus\bigcup_{j<2}N_{\delta,j}$.
Hence, $\rho(b_{\af,i}^m\wedge h_{\af,i},N)=\delta$.

Second, we have,
by Corollary~\ref{rcnested} and Claim~\ref{roundup},
$$\pi_+^0(b_{\af,1}^m\wedge h_{\af,1},N)
=\bigwedge_{j<2}\pi_+^{N_{\delta,0}}(\pi_+^j(b_{\af,1}^m\wedge h_{\af,1},M))
=\pi_+^{N_{\delta,0}}(-b_{\af,0}^{m+1})\wedge\pi_+^{N_{\delta,0}}(b_{\af,1}^m).$$
We have $\pi_+^{N_{\delta,0}}(-b_{\af,0}^{m+1})=-b_{\af,0}^{m+1}$ because
$-b_{\af,0}^{m+1}\in N_{\delta,0}$ because 
$\rho(-b_{\af,0}^{m+1},M)\in I_0(\af)=\beta$. We have
$\pi_+^{N_{\delta,0}}(b_{\af,1}^m)=\pi_+^{A'_{\af,2}}(b_{\af,1}^m)=1$ 
by Lemma~\ref{projrestrict} because 
$N_{\delta,0}\cap\Omega=A_{\af,0}\commute A'_{\af,1}$
by \eqref{oldacomm} and Proposition~\ref{communion},
and because $A_{\af,0}\cap A'_{\af,1}=A'_{\af,2}\subrc A'_{\af,1}$.
Thus, $\pi_+^0(b_{\af,1}^m\wedge h_{\af,1},N)=-b_{\af,0}^{m+1}$.

Third, we have,
by Corollary~\ref{rcnested} and Claim~\ref{roundup},
$$\pi_+^1(b_{\af,0}^m\wedge h_{\af,0},N)
=\bigwedge_{j<2}\pi_+^{N_{\delta,1}}(\pi_+^j(b_{\af,0}^m\wedge h_{\af,0},M))
=\pi_+^{N_{\delta,1}}(b_{\af,0}^m)\wedge\pi_+^{N_{\delta,1}}(-b_{\af,1}^m).$$
We have $\pi_+^{N_{\delta,1}}(-b_{\af,1}^m)=-b_{\af,1}^m$ because
$\rho(-b_{\af,1}^m,M)\in I_1(\af)=[\beta,\af)\subseteq N_{\delta,1}$.
We have $\pi_+^{N_{\delta,1}}(b_{\af,0}^m)=\pi_+^{A'_0\cap N_{\delta,1}}(b_{\af,0}^m)$
by Lemma~\ref{projrestrict} because 
$A'_{\af,0}\commute(\Omega\cap N_{\delta,1})$
by \eqref{oldacomm} and Proposition~\ref{communion},
and because, arguing as in the
proof of Claim~\ref{rcdir}, $A'_{\af,0}\cap N_{\delta,1}\subrc A'_{\af,0}$.
Because $A'_{\af,0}\cap N_{\delta,1}=A'_{\af,2}$ by Claim~\ref{reflect},
we also have $\pi_+^{A'_0\cap N_{\delta,1}}(b_{\af,0}^m)=\pi_+^{A'_{\af,2}}(b_{\af,0}^m)=1$.
Thus, $\pi_+^1(b_{\af,0}^m\wedge h_{\af,0},N)=-b_{\af,1}^m$.

Thus, $\Omega$ has the FN but not the SFN.

We briefly remark that the interaction between 
$\pi_+^0$ and $\pi_+^1$ is essential to
the above construction. Given boolean algebras
$K\subrc L$, it is not hard to check, using
Lemma~\ref{rcmeet}, that 
$(L,\wedge_L,\vee_L,{-}_L,0_L,1_L,\pi_+^K,\pi_-^K)$
is locally finite. This lemma can also be used 
to re-prove the implication from FN to SFN 
for boolean algebras of size at most $\al_1$, 
without using, as Heindorf and Shapiro do, 
the implication from FN to projectivity 
for boolean algebras of size at most $\al_1$.

Indeed, given a boolean algebra $A$ of size at most $\al_1$
with the FN, let $(M_\af)_{\af<\oml}$ be a \lapps{\oml}
with $A\in M_0$, let $\mc{F}_1$ be a chain of
finite subalgebras of $A\cap M_0$ with union $A\cap M_0$,
and then inductively assume that
$1\leq\af<\oml$ and $\mc{F}_\af$ is a pairwise commuting 
cofinal family of finite subalgebras of $A\cap M_{\af,0}$.
Let $A\cap M_\af=\{a_n:n<\om\}$ and set $C_0=\{0,1\}$.
Given $n<\om$ and a finite $C_n\subalg A\cap M_\af$,
let $A_n$ be the subalgebra of $A$ generated by $C_n\cup\{a_n\}$; 
choose $B_n\in\mc{F}_\af$ containing $\pi_+^{M_{\af,0}}[A_n]$;
let $C_{n+1}$ be the subalgebra of $A$ generated by $A_n\cup B_n$. 
By Lemma~\ref{rcmeet}, $\pi_+^{M_{\af,0}}[C_{n+1}]=B_n$; hence, 
for all $D\in\mc{F}_\af$, 
$\pi_+^D[C_{n+1}]=(\pi_+^D\circ\pi_+^{M_{\af,0}})[C_{n+1}]=\pi_+^D[B_n]$, 
which implies $C_{n+1}\commute D$ by Lemma~\ref{commproj}.
Thus, $\mc{F}_{\af+1}=\mc{F}_\af\cup\{C_n:n<\om\}$
is a pairwise commuting cofinal family 
of finite subalgebras of  $A\cap M_{\af+1,0}$.

\end{document}